\documentclass[11pt, leqno]{article}
\usepackage{amsmath,amsthm}
\usepackage{amsfonts}
\usepackage{mathrsfs}
\usepackage{amssymb}
\usepackage{enumitem}
\usepackage[usenames,dvipsnames]{pstricks}
\usepackage{epsfig}
\usepackage{esint}    
\usepackage{pst-grad} 
\usepackage{pst-plot} 
\usepackage{pinlabel}
\usepackage{graphicx}
\usepackage{color}
\usepackage{epstopdf}
\usepackage{hyperref}

\newtheorem{theor}{Theorem}[section]
\newtheorem{cor}[theor]{Corollary}
\newtheorem{lemma}[theor]{Lemma}
\newtheorem{prop}[theor]{Proposition}

\newtheorem*{question*}{Question}

\theoremstyle{definition}
\newtheorem{defn}{Definition}

\theoremstyle{remark}

\numberwithin{equation}{section}
\numberwithin{defn}{section}

\newcommand{\R}{\mathbb{R}}        

\newcommand{\RN}{\mathbb{R}^N}

\newcommand{\eps}{\varepsilon}
\newcommand{\dt}{\mathrm{d}t}
\newcommand{\ve}{\varepsilon}

\DeclareMathOperator{\Lap}{(-\Delta)} 

\newcommand{\Ds}{\Lap^{s}}
\newcommand{\Dsi}{\Lap^{-s}}
\newcommand{\Dsm}{\Lap^{s/2}}
\newcommand{\Dsd}{\Lap^{s_\delta}}
\newcommand{\Doh}{\Lap^{1/2}}
\renewcommand{\d}{\mathrm{d}}

\usepackage{color}
\usepackage[normalem]{ulem} 

\definecolor{darkblue}{rgb}{0.05, .05, .65}
\definecolor{darkgreen}{rgb}{0, .6, .2}
\definecolor{darkred}{rgb}{0.8,0,0}

\begin{document}

\title{ \bf Non-existence and instantaneous \\extinction of solutions for singular \\ nonlinear fractional diffusion equations}
\author{by Matteo Bonforte$^{\rm a,b}$, Antonio Segatti$^{\rm d,e}$, Juan Luis V\'azquez$^{\rm a,c}$}
\date{}

\maketitle

\begin{abstract}
We show non-existence of solutions of the Cauchy problem in $\RN$ for the
nonlinear parabolic equation involving fractional diffusion
$\partial_t u + \Ds \phi(u)= 0,$ with $0<s<1$  and very singular nonlinearities $\phi$ . More precisely, we prove that when  $\phi(u)=-1/u^n$ with $n>0$, or $\phi(u) = \log u$, and  we take nonnegative $L^1$ initial data, there is no (nonnegative) solution of the problem in any dimension $N\ge 2$. We find the range of non-existence when $N=1$ in terms of $s$ and $n$.  The range of exponents that we find for non-existence both for parabolic and elliptic equations are optimal. \normalcolor
Non-existence is then proved for more general nonlinearities $\phi$, and it is also extended to  the related elliptic problem of nonlinear nonlocal type: $u + \Ds \phi(u) = f$ with the same type of nonlinearity  $\phi$.
\end{abstract}

\vskip 1cm

\\

\vfill

\begin{itemize}
[leftmargin=5mm]\setlength\itemsep{-2pt}
\item[(a)] Departamento de Matem\'{a}ticas, Universidad
Aut\'{o}noma de Madrid,\\ Campus de Cantoblanco, 28049 Madrid, Spain
\item[(b)] e-mail address:~\texttt{matteo.bonforte@uam.es }\\ web-page:~\texttt{http://www.uam.es/matteo.bonforte}
\item[(c)] e-mail address:~\texttt{juanluis.vazquez@uam.es }\\ web-page:~\texttt{http://www.uam.es/juanluis.vazquez}
\item[(d)] Dipartimento di Matematica ``F. Casorati'' Universit\'a di Pavia,\\ Via Ferrata 1, 27100 Pavia, Italy
\item[(e)] e-mail address:~\texttt{antonio.segatti@unipv.it}\\ web-page:~\texttt{http://www-dimat.unipv.it/segatti}
\end{itemize}

\newpage

\section{Introduction}

In this paper we discuss the question of non-existence of solutions of a class of nonlinear parabolic equations involving fractional diffusion and singular nonlinearities of the so-called ultra-fast diffusion type. These equations have the form
\begin{equation}\label{geq.Phi}
\partial_t u + \Ds \phi(u)= 0,
\end{equation}
where $\phi:\R_+\to \R$ is a monotone nondecreasing function of $u$ with
a singularity in $u=0$, namely such that $\phi(0+)= -\infty$. Consequently, nonnegative data and solutions are considered. On the other hand,  $\Ds$ is the fractional Laplacian operator
with $0<s<1$, that is the nonlocal operator defined,
at least for functions in the Schwartz class $\mathcal{S}$, by
\begin{equation}
\label{eq:fract_lap}
\Ds v (x) = c(N,s)\, \hbox{p.v.}\int_{\RN}\frac{v(x) - v(y)}{\vert x-y\vert^{N+2s}}\d y, \,\,\,\forall x\in \RN\,,
\end{equation}
see, e.g., \cite{land}. Here, $\hbox{p.v.}$ stands for principal value of the integral and $c(N,s)$ is a positive scaling constant, whose exact value is not important throughout this paper.

We consider the Cauchy problem posed in the whole space with initial data
\begin{equation}\label{id}
u(x, 0) = u_0(x)\,\,\, \hbox{ for } x\in \RN, \quad N\ge 1\,,
\end{equation}
and we assume that $u_0$ is nonnegative and integrable, which corresponds to the natural conditions assumed for a mass distribution that evolves in time by nonlinear nonlocal diffusion. Accordingly, we look for nonnegative solutions that are integrable in space for every time $t>0$.

The limit case $s=1$ of equation \eqref{geq.Phi} corresponds to the well-known model \color{darkblue} of \normalcolor nonlinear diffusion driven by the standard Laplacian, $\partial_t u =\Delta \phi(u)$, sometimes called Filtration Equation. This equation has been studied by numerous authors in the last century, see the early works \cite{BeTh72, CL71} and the monographs \cite{Vapme, JLVSmoothing} and their references. The well developed theory implies that given $u_0$ as above, if $\phi$ is continuous, strictly increasing and $\phi(0)=0$,\footnote{Actually, less stringent conditions on $\phi$ are acceptable.} then there exists a unique solution of the Cauchy problem, where solution is understood in some suitable generalized sense like mild solution. \normalcolor Moreover, the operator that assigns to
any initial condition its evolution is a semigroup of contractions
in $L^1_+(\RN)$ (or in $L^1(\RN)$ if two-signed solutions are considered). The set of admissible $\phi$ includes, in particular, all powers $\phi(u)=u^m$ for $m>0$, in which case the theory is more detailed. The use of mild solutions is usually changed into other familiar concepts, like weak or strong solutions, with better functional properties. Uniqueness theorems are proved to make such concepts useful.

The extension of this theory to the fractional Laplacian model \eqref{geq.Phi} has been done recently in a series of papers \cite{pqrv1, pqrv2, pqrv3, VPQR13} for convenient choices of the nonlinearity $\phi$. The first two references treat the case where $\phi(u)$ is a positive power and construct a contraction semigroup with good properties, the third treats a case of logarithmic nonlinearity, while the last one treats the case of a more general smooth function $\phi$, and then proves that when $\phi'(u)>0$ for all $u$ (i.\,.e., $\phi$ is non-degenerate) the nonnegative solutions are indeed positive and $C^\infty$ smooth.  The quantitative analysis of positive solutions was then pursued in \cite{BV2012} in the form of regularity estimates, together with the study of existence and uniqueness of initial traces.

\medskip

\noindent {\sc Non-existence in singular diffusion.} The purpose of this paper is to show a phenomenon of non-existence of solutions of the Cauchy Problem \eqref{geq.Phi}-\eqref{id} when the nonlinearity $\phi$ is of singular type, more precisely when $\phi(u)$ is defined and increasing (and possibly smooth) for $u>0$, but $\phi(0+)=-\infty$. This forces us to work only with nonnegative solutions, a restriction that is kept throughout in the paper. This type of singular nonlinearities were first considered for the Filtration Equation and
then in a number of physical applications.
The prototypical choice for the nonlinearity is an inverse
power, that is $\phi(u)=cu^m$ for $m=-n\le 0$. As an example, when $m=-1$,
the equation rules the evolution of the temperature in the Penrose-Fife model for phase transition (see \cite{PeFe}
and \cite{SSZ12} and the references therein).

The first caveat is that one has to change the sign of the power, $c=-a<0$ if we want the resulting singular equation to be formally parabolic,
for then we can write
$\partial_t u = \Delta\phi(u)$  in the form
\begin{equation}\label{ufdq}
\partial_t u =\mbox{div}\,(D(u)\nabla u)\,, \quad \mbox{with} \ D(u)=anu^{-(n+1)}>0.
\end{equation}
Since the diffusivity $D(u)$ is infinite at $u=0$ we talk about singular diffusion or fast diffusion, and for $n>0$ (i.e., $m<0$) we call it  {\sl ultra-fast} diffusion or \textsl{very singular} diffusion, cf. \cite{JLVSmoothing}. \normalcolor There is no loss of generality in assuming that $an=1$ to simplify the form of the equation. The second caveat is that the case $m=0$ enters into this theory if we put $\phi(u)=\log(u)$, since then $D(u)=1/u$, cf. \cite{JLVSmoothing} and its references.

Very singular fast diffusion equations have curious properties, like the non-existence of solutions for reasonable classes of data that we are going to prove. The present work is motivated by the result  of \cite{Vaz92}, where it was  found in that no  solutions exist for the Cauchy Problem for equation \eqref{ufdq}  with nonnegative, integrable initial data in the following cases: if $N\ge 3$ for $m\le 0$, or if $N=2$ for $m<0$, or if $N=1$ for $m\le -1$, and besides the range is optimal (note that we write $m=-n$). And the nonexistence applies to any time interval, however small, we have a phenomenon of {\sl instantaneous extinction} or {\sl extinction in zero time}. It must be noted that the problem with existence occurs exclusively because of the singular level $u=0$. Thus, solutions of this problem do exist for data that do not go to zero at infinity, or even for data in some $L^p(\RN)$, on the condition of not having a rapid decay as $|x|\to\infty$. Another situation in which solutions exist is the case of the equation posed in a bounded domain coupled with zero Neumann conditions or
even with dynamic boundary conditions (see \cite{SSZ13}).

One then wonders if the presence of fractional diffusion will affect the non-existence result, since in particular we cannot identify a diffusivity as the $D(u)$ in the standard Laplacian case; if the answer is yes, we want to identify  the techniques that enable us to  treat the problem and get optimal results. We will give answers to such questions below. Briefly stated, the phenomenon of non-existence of solutions still holds, and we find that the best approach to the proof is via limits of non-degenerate regularizations and the use of Riesz potentials. Moreover, when $\phi(u)$ has the form of an inverse power with exponent $n$, we get the optimal range of exponents $n$ for every dimension $N$ and fractional exponent $s$. And we extend the result to more general nonlinearities, though in a less sharp formulation.  We also obtain sharp non-existence results for elliptic problems.

The next section contains a statement of the problem, the technique of approximation by regularization, the concept of limit solution, and the statement of the main results with optimal ranges of application. Further outline of the contents of this paper will be given at the end of that section.

\section{Problem, limit solutions and main results}

We discuss the question of non-existence for the following class
of diffusion equations
\begin{equation}\label{eq:FrUF}
\begin{cases}
\partial_t u + \Ds {\phi(u)} = 0,\,\,\, \hbox{ in } Q\,,  \\
u(x, 0) = u_0(x)\,\,\, \hbox{ for } x\in \RN\,,
\end{cases}
\end{equation}
where  $Q:=\RN\times (0,T)$, with $T>0$ and $N\ge 1$. The nonlinear function is defined, increasing and smooth for $u>0$ with $\phi(s)\to -\infty$ as $s\to0$. More precisely, in this section and in most of the paper $\phi$  chosen from the list
\begin{equation}
\label{eq:nonlinearity}
\phi_n(u):=
\begin{cases}
-u^{-n},\,\,\,\hbox{ for } n>0,\\
\log u \,\,\,\,\hbox{ for } n =0.
\end{cases}
\end{equation}
We remind the reader that we will use throughout the notation $N$ for
the space dimension, and $n$ for the nonlinearity exponent ($m=-n$ is the more standard notation used when $\phi$ is not singular, \cite{Vapme}, but using $m<0$  is rather inconvenient).
The reason for which the case $m=0$ is rewritten in terms of the logarithm
nonlinearity is well-documented for classical diffusion $s=1$, where the equation
$$
\partial_t u =\Delta (\log u)\,,
$$
can be rigorously seen as the limit as $m\searrow 0$ of the  porous
medium equation
$$
\partial_t u +\Ds (u^m/m) = 0,\,\,\,\,m>0.
$$
cf. \cite{rev97}, \cite[Chapter 8]{JLVSmoothing} and also \cite{BF2012}. For $0<s<1$ similar results are obtained in \cite{pqrv3} and in the recent paper \cite{V-exist} for $N=1$.\normalcolor

Following \cite{Vaz92}, the strategy of proof of our nonexistence results is based on approximating problem
\eqref{eq:FrUF} with data $u_0\in L^1(\RN)$ by the family problems
\begin{equation}
\label{eq:FrUF_approx}
\begin{cases}
\partial_t u_\eps  + \Ds ({\phi(u_\eps)}) = 0, \qquad n>0,\\
u_\eps(x, 0) = u_0(x) + \eps\,\,\, \hbox{ for } x\in \RN,
\end{cases}
\end{equation}
so that we avoid data with values on the singular level $u=0$, so that any $\eps>0$ consider the
sequence $u_{0,\eps}:=u_0 +\eps$. Then, the standard theory must apply, see next section, and a classical solution $u_\ve(t,x)$  exists for all $\ve>0$ and $u_\ve\ge \ve$. Moreover, the maximum principle holds for these classical solutions and we have $u_\eps\ge u_{\eps'} $ for $\eps\ge \eps'>0$. Therefore, we will be allowed to take the monotone limit
\begin{equation}
\bar u(x,t)=\lim_{\eps\to 0} u_\eps(x, t)\,,
\end{equation}
see the proof of Proposition \ref{prop.limit} below for a rigorous justification. This function is a kind of generalized solution of the problem, that we will call the {\sl limit solution.} It is now an important step of the theory to decide in which  sense this is a solution of the equation in a more traditional functional sense (like weak, strong or viscosity solution) and also in which sense it takes the initial data. In cases of non-uniqueness of such solutions, the unique limit obtained by the above method has been called by various names: maximal solution, SOLA, proper solution,...

There is an extreme possibility that we will discuss in this paper:  in the limit $\ve\to0$ the family $u_\ve$ might converge to zero for any  $t>0$ and all $x\in\RN$, namely  $\bar u(x,t)\equiv 0$, and this is the plain statement of the non-existence theorem. For clarity of presentation, we split the results in the following two Theorems.

\begin{theor} \label{th:non ex}
Let $N\ge 2$, $n\ge 0$, and $0<s<1$. Then the limit solution $\bar u$ of \eqref{eq:FrUF} corresponding to any initial data any $u_0\in L^1_+(\RN)$ vanishes identically for all $x\in\RN$, $t>0$. More precisely, we have
\begin{equation}
\lim_ {\eps\to 0} u_\eps (x,t)=0,
\end{equation}
and the convergence is uniform in $\RN\times (\tau,\infty)$ if $u_0$ is bounded, while it happens  locally in $L^1(Q)$ for general $u_0\in L^1_+(\RN)$. More  precisely, for every ball $B\subset \RN$
\begin{equation}
\lim_{\eps\searrow 0}u_{\eps} = 0 \hbox{ in } L^1(B) \,\,\,\hbox{ uniformly in }t\ge \tau \hbox{ for any }\tau>0\,.
\end{equation}
\end{theor}

\medskip

The one-dimensional case is worth a separate statement since there appears a critical relation,  $2s-n=1$, to determine the range of existence or non-existence. It will require a special analysis.

\begin{theor}\label{th:non_ex_1D} Let $N\ge 1$, $n\ge 0$, and $0<s<1$. Then the limit solution $\bar u$ of \eqref{eq:FrUF} corresponding to any initial data for any $u_0\in L^1_+(\R)$ vanishes identically in the sense of the previous Theorem unless $1/2<s<1$ and $n<2s-1$, or $s=1/2$ and $n=0$.
This range of exponents is optimal.
\end{theor}
The non-existence picture in one dimension is given in Figure \ref{fig.1}.
\begin{figure}[!]
\centering
    \includegraphics[width=15cm]{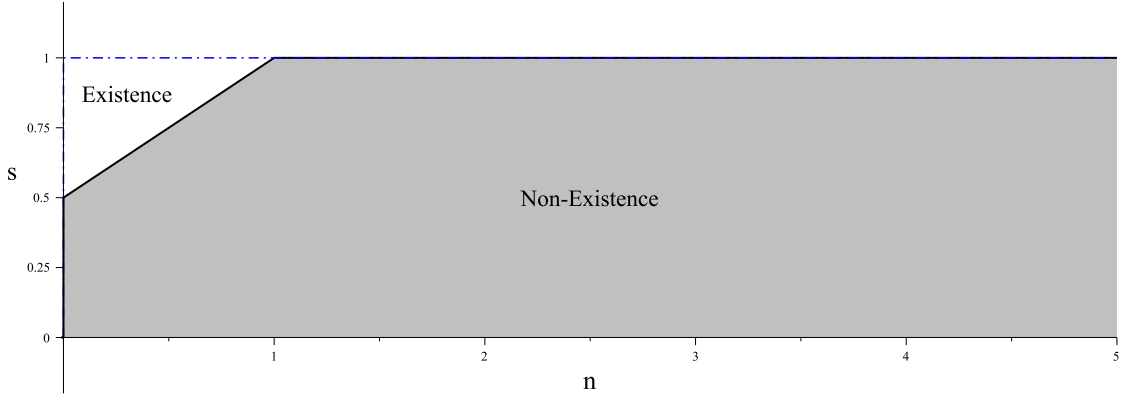}
 \caption{\noindent  \textbf{The non-existence picture in one dimension. }\textit{The white upper triangular region, corresponds to the range of parameters where there is existence:  $n<2s-1$ and $1/2<s\le 1$, together with the special point $n=0$ and $s=1/2$. The grey region corresponds to the range of parameters where there is no existence.}}
  \label{fig.1}
\end{figure}

\noindent\textbf{Remarks. }
\noindent(i) It is important to notice that the function $\overline{u}\equiv 0$ in $\RN\times (\tau, +\infty), \forall \tau>0$ cannot be a solution of the equation in any differential sense since  $\phi(\bar u) \equiv -\infty$ for any point of $\RN\times (\tau, +\infty), \forall \tau>0$. On the other hand, the initial value has collapsed from $u_0(x)+\eps$ to zero at $t=0+$, and not to $u_0(x)$ as expected. This is called instantaneous extinction, and  shows a strong discontinuity of the semigroup at  $t=0$, for data $u_0\in L^1_+(\RN)$.

\noindent(ii) The result depends in a crucial way on the class of initial data. Notice that positive constants are indeed classical solutions to this problem and of course they do not extinguish. Therefore, there should be an intermediate decay between $L^1$ data and constants, such that nontrivial solutions with that decay at infinity will still exist. For standard diffusion $s=1$, the critical decay that implies existence has been investigated by Daskalopoulos and del Pino in the series of papers
\cite{DDP1994, DDP1995, DDP1997} in which they provide an optimal decay condition to determine the existence or the immediate extinction in classes of $L^p$ data.

\noindent(iii) Technically, the proofs of the above theorems require a different treatment, according to whether $N>2s$ or not. The critical situation happens when $N=2s$, hence there appear two different types of behaviour only in the one dimensional case. Indeed, for all dimensions $N\ge 2$ we always ve $N>2s$ and this is covered in Proposition \ref{th:noex1}; and the same proof allows also to  show non-existence the case $N=1$ and $0<s<1/2$.

\medskip

Once we have proved that the approximate solutions converge to zero, the natural question is what happens to solutions in a more traditional sense, like weak, strong or viscosity. To answer this question, we first we have to introduce the concept of solutions that we are going to use.
\begin{defn}
\label{def:weak_sol}
A function $u$ is a \textsl{weak (energy) solution} of equation \eqref{geq.Phi} if it is nonnegative and
\begin{itemize}
\item $u\in C^0(0,+\infty;L^1(\RN))$, ${ \phi(u)}\in L^2_{loc}(0,+\infty;\dot H^{s}(\RN))$
\item it satisfies, for all $\zeta\in C^\infty_{0}(Q)$,
\begin{equation}
\label{eq:weak_form1}
\int_{0}^{\infty}\int_{\RN} u \partial_t\zeta \d x\d t - \int_{0}^{\infty}\int_{\RN}
\Dsm { \phi(u)} \Dsm \zeta \d x\d t = 0.
\end{equation}
\item A weak energy solution is said to be a \textsl{strong solution} if moreover $u\in C((0,+\infty);L^1(\RN))$ and $\partial_t u(t,\cdot)\in L^1(\RN)$ for every $t\ge \tau>0$ and every $\tau>0$.
\end{itemize}
\end{defn}
We next establish a comparison result between approximate and strong solutions, namely we show in Theorem \ref{th:comparison} that $u\le u_\varepsilon$ where $u$ is a strong solution and $u_\varepsilon$ is the above mentioned approximate solution. More details are given in Section \ref{sec.nonweak}, where we prove the following theorem.
\begin{theor}\label{th:nonex-strong}
Let $u$ denote a nonnegative strong solution of \eqref{eq:FrUF}, in the range of parameters $(N,s,n)$ allowed in Theorems $\ref{th:non ex}$ and $\ref{th:non_ex_1D}$. Then, $u\equiv 0$ in $\RN\times (\tau, +\infty)$ for any $\tau>0$. This range of exponents is optimal.
\end{theor}
We shall mention that once we establish our non-existence result for strong solutions, then the same result will hold for any reasonable limit of strong solutions.
Moreover, the Theorem above entails that no Barenblatt solution exists for the singular equation
\eqref{eq:FrUF}. More generally, we can conclude that
there exist no solution (in a reasonable differential sense) to \eqref{eq:FrUF} that regularizes into $L^1(\RN)$ for positive times.

\medskip

\noindent {\sc Outline of the rest of the paper. Further results. }We establish the existence of solutions of the approximation problems and their properties in Section \ref{sec.approx}, and pass to the limit to obtain a so-called limit solution. \\
We then prove the non-existence results, in the form of vanishing limit solution, first for $N>2s$, Section \ref{sec.proofs}, and then for the remaining cases in 1D, Section \ref{sec.1D}. We show optimality of the exponents in 1D in Section \ref{sec.exist.1D}.
Sections \ref{sec.proofs}, \ref{sec.1D} and \ref{sec.exist.1D} contain the proof of Theorems \ref{th:non ex} and \ref{th:non_ex_1D}.\\
Section 7 contains the proof of above mentioned comparison result between strong and approximate solutions, together with the proof of Theorem \ref{th:nonex-strong}.

\medskip

In later sections  we deal with two interesting extensions of the results.

\noindent$\bullet$ \textit{More general nonlinearities. } Using concentration comparison and rearrangement techniques developed in \cite{VazVol1, VazVol2}, we prove that the results of Theorems \ref{th:non ex} and \ref{th:non_ex_1D} can be \textsl{extended to more general nonlinearities}, as stated in the Theorem \ref{Thm.non_ex-Symmetriz}; we refer to Section  \ref{sec.comparison.nonlin} for more details.

\medskip

\noindent$\bullet$  \textit{Elliptic equations. } An important extension of our results concerns the question of existence for solutions to \textsl{elliptic equations} of the form
\begin{equation}
u+\Ds \phi(u)=f\,,\qquad\mbox{with $f\in L^1_+(\RN)$}\,.
\end{equation}
A complete study is performed in Section \ref{sec.elliptic} where we show how the parabolic technique developed in Sections \ref{sec.proofs}, \ref{sec.1D} and \ref{sec.exist.1D} can be adapted to treat this problems and obtain analogous non-existence results, that are stated in  Theorem \ref{th:non_ex_elliptic}. In particular, the optimal ranges of non-existence of the elliptic and parabolic problems coincide.

\medskip

There is another section devoted to the Dirichlet problem on bounded domains. This is rather straightforward, once we make sure that a suitable comparison holds.

\section{Study of the approximate solutions }\label{sec.approx}

It is convenient to write $u_\eps(x,t)=v_\eps(x,t) +\eps$ and then try to solve
 the Cauchy problem
\begin{equation}
\label{eq:FrUF_approx1}
\begin{cases}
\partial_t v + \Ds (\phi_\eps(v)) = 0\,\,\,\,\,\hbox{ with }\phi_\eps(v):= \phi(v+\eps) -\phi(\eps)
\\
v(0) = u_0 \,\,\,\,\hbox{ for } x\in \RN.
\end{cases}
\end{equation}
for all $\eps>0$. This is a modified problem prepared to avoid the degenerate level of the equation by displacement of the axes. Note that for $\eps>0$  and for nonnegative arguments $\phi_\eps$ is a smooth, positive, monotone increasing function with $\phi_\eps(0) = 0$.

Let us construct the approximate solution $u_\eps$ to \eqref{eq:FrUF_approx}. It will be convenient in a first approximation to assume that $u_0$ is bounded.

\medskip

\noindent {\bf Existence and properties.} The theory of existence and uniqueness of weak solutions to \eqref{eq:FrUF_approx1} is given in \cite[Theorem 8.2]{VPQR13}. For any $\eps>0$, let $v_\eps$ denote the weak solution to \eqref{eq:FrUF_approx1}.
 Here weak solution means a function
$v\in C^0([0,+\infty);L^1(\RN))$ such that $\phi_\eps(v)\in L^2_{loc}(0,+\infty;\dot{H}^{s}(\RN))$
satisfying $v(0) = u_0 \hbox{ a.e. in }\RN$ and
\begin{equation}
\label{eq:weak_form}
\int_{0}^{\infty}\int_{\RN} v \partial_t\zeta \d x\d t - \int_{0}^{\infty}\int_{\RN}
\Dsm (\phi_\eps(v)) \Dsm \zeta \d x\d t = 0
\end{equation}
for any $\zeta\in C^{\infty}_{c}(Q)$.  After obtaining these solutions  we restore  for any $\eps>0$ the original $u$-level by defining
$u_\eps:=v_\eps +\eps$, as stated at the beginning.
Clearly, we have that $u_\eps\ge \ve$ (actually, $u_\eps> \ve$) and
hence $v_\eps\ge 0$ in $Q$.

Let us list some further properties of $v_\ve$ and $u_\ve$.

\noindent $\bullet$ {\sl Boundedness and regularity.} These  solutions are shown to be bounded for strictly positive times. More precisely, for every $t>0$ and every $p\in [1,\infty]$ there holds
\begin{equation} \label{eq:Lp}
\|v_\eps(\cdot,t)\|_{L^p(\RN)}\le \|u_0\|_{L^p(\RN)}
\end{equation}
As a consequence,  if $u_0$ belongs to $L^\infty$, then $v_\eps$ is regular enough to satisfy the equation in the classical sense at least when $t>0$\,, by the results of \cite{VPQR13}. Therefore, $u_\eps=v_\varepsilon+\varepsilon$ is smooth and satisfies the original equation in the classical sense in $ Q$. Under these circumstances, the initial data are also taken, at least in the sense of convergence in $L^1(\RN)$.

\noindent $\bullet$ {\sl Conservation of the $L^\infty$ bound. } If $u_0$ belongs to $L^\infty$, then we can remove the boundary layer at $t=0$. More precisely, we can proceed as in \cite{pqrv2,VPQR13} to conclude that any classical
solution of \eqref{eq:FrUF_approx} and of \eqref{eq:FrUF_approx1} satisfies the following
\begin{equation}
\label{eq:li_decay}
\|v_\eps(t)\|_{L^\infty(\RN)}\le \|v_\eps(\tau)\|_{L^\infty(\RN)}
\,\,\,\forall 0\le\tau<t\,.
\end{equation}

\noindent $\bullet$ {\sl Mass conservation. }Nonnegative solutions to the evolution equation \eqref{eq:FrUF_approx1} conserves the mass, cf. \cite{pqrv2,VPQR13}. More precisely, we have for all $t\ge 0$
\begin{equation}
\label{eq:mass_cons}
\int_{\RN} v_\eps(x,t)\d x = \int_{\RN} u_0(x)\d x\qquad \mbox{i.\,e.,}\qquad
\int_{\RN} (u_\eps(x,t)- \eps)\d x = \int_{\RN} u_0(x)\d x
\end{equation}

\noindent $\bullet$ {\sl $L^1$-contraction and comparison. }The evolution \eqref{eq:FrUF_approx1} is an $L^1$ contraction, namely
we have
\begin{equation}\label{eq:contraction}
\int_{\RN}(u_\eps(t) -z_\eps(t))_{+}\d x \le\int_{\RN}(u_0 -z_0)_{+}\d x \,\,\,\hbox{ for }t>0,
\end{equation}
where $z_\eps$ is a solution to \eqref{eq:FrUF_approx} constructed starting from another data $z_0\in L^1(\RN)$.
Here and in what follows $(\cdot)_+$ denotes the positive part function. For the proof we may again refer to  \cite{pqrv2,VPQR13}.

\noindent $\bullet$ {\sl Monotonicity with respect to $\eps$. } An easy version of the above comparison argument shows also that for $0<\ve<\ve'$ we have $0<\ve\le u_\eps\le  u_{\ve'}$.

\noindent $\bullet$ {\sl Time monotonicity.}
The scaling invariance of the equation entails the following important mononoticity
property of the strong solutions, known as the Aronson-B\'enilan inequality,
\begin{equation}\label{eq:AB}
\partial_t u_\eps \le \frac{u_\eps}{(n+1) t}\,\,\,\,\,\forall (x,t)\in Q.
\end{equation}

\noindent {\sl Proof.~} This is rather classical in the PMR theory, \cite{Vapme}, but we give here a proof for convenience of the reader. This inequality is a consequence of the scaling properties of the equation.
We give a brief sketch of the argument that lead to \eqref{eq:AB} (for more
informations see \cite{Vapme}).
Starting from a smooth solution $u$ of \eqref{eq:FrUF} and a parameter $\lambda>1$,
we construct the function $\tilde u_{\lambda} (x,t):= \lambda u(x, \lambda^{-n}t)$.
The scaling invariance of the equation gives that $\tilde u$ is again a
solution of \eqref{eq:FrUF} with initial condition $\tilde u_\lambda(x,0) = \lambda u_0(x)$.
In particular, note that since $\lambda >1$, we have that $\tilde u_\lambda(x,0) \ge u_0(x)$
and thus, thanks to the maximum principle, we have that $\tilde u_\lambda \ge u$ in $Q$.
An important consequence of this fact is that the right derivative
 $\frac{d}{d\lambda}_{|\lambda =1} \tilde u_\lambda \ge 0$
 (this inequality can be proved
 by working directly on the incremental quotient). Thus, we have
 $$ 0\le \frac{\rm d}{\rm d\lambda}_{|\lambda =1} \tilde u_\lambda = u - (n+1)t \partial_t u\,.
 $$\qed

\noindent{\bf Passing to the limit.} We may now pass to the limit $\eps\to0$ using the monotonicity of the family $u_\eps$. When  $u_0$ is bounded the limit is taken in a set of uniformly bounded functions, hence by the monotone convergence theorem the limit $\bar u$ is taken in
the local $L^1$ sense, i.e., in  $L^1(B)$ for every compact subset of $\RN\times[0,\infty]$.
When $u_0$ is not bounded but still integrable, we use the $L^1(\RN)$ contractivity result to obtain the same result. Summing up,

\begin{prop}\label{prop.limit} If $u_0$ is a nonnegative function in $L^1(\RN)$ there exists the monotone limit
$\bar u=\lim\limits_{\eps\to 0} u_\eps$ with local convergence in $L^1(Q)$.
\end{prop}


\section{Proofs of non-existence. The main case}\label{sec.proofs}

We will establish here the trivial limit of the approximation scheme in the range $N>2s>0$,
of course for nonnegative data $u_0\in L^1(\RN)$ and exponents $n\ge 0$. This subsection is the core of our technical argument. An important feature in our proof is that the inverse operator of $\Ds$ can be expressed as a convolution with the Riesz potential
(see e.g.\cite{stein}), namely
\begin{equation}
\label{eq:riesz}
(-\Delta)^{-s} u(x) = c_{N,s}\int_{\RN}\frac{u(x-y)}{\vert y\vert^{N-2s}}\d y\,,
\end{equation}
where $c_{N,s}$ is a positive normalizing constant. The above representation does not hold in the subcritical case $N-2s\le 0$. This happens in dimension one and for $s\in [1/2,1)$, and requires a different argument, cf. Section \ref{sec.1D}.

We want to show that the sequence $u_\eps$ of solutions to \eqref{eq:FrUF_approx}
converges to zero as $\eps\searrow 0$ for strictly positive times. The following result sums up the situation.
\begin{prop}\label{th:noex1}
For any given $u_0\ge 0$ with  $u_0\in L^1(\RN)$, let
$u_\eps$ be the solution of \eqref{eq:FrUF_approx}.
Then for all $n\ge 0$ and  $N> 2s$ we have
\begin{equation}
\label{eq:no_ex}
\lim_{\eps\searrow 0} u_\eps = 0
\end{equation}
and the convergence is uniform in $\RN\times (\tau,\infty)$ if $u_0$ is bounded (the time $\tau>0$ is arbitrary), while it happens locally in $L^1(Q)$ for general $u_0\in L^1_+(\RN)$. More  precisely, for every ball $B\subset \RN$
\begin{equation}
\lim_{\eps\searrow 0}u_{\eps} = 0 \hbox{ in } L^1(B) \,\,\,\hbox{ uniformly in }t\ge \tau \hbox{ for any }\tau>0\,.
\end{equation}
\end{prop}
Notice that this result gives the proof of Theorem \ref{th:non ex}, i.e. applies to dimensions $N\ge 2$ for all $0<s<1$ and all $n\ge 0$. However, in $N=1$ we only cover the case $0<s<1/2$, therefore we prove Theorem \ref{th:non_ex_1D} only when $2s<1$. The whole picture in 1D depends on other arguments and will be treated in the next section.

\subsection{Proof of Proposition \ref{th:noex1}}The proof is divided in several steps. We will begin by considering the case of initial data that are smooth and with compact support (or decay fast at infinity). Finally, in a the last step, using an approximation/density argument we will cover
the general case of $L^1_+(\RN)$ initial conditions.

\noindent$\bullet~$\textsc{Step 1. }\textit{Proof of convergence for good data. Approximate solutions. }Consider the sequence $u_\eps$ of solutions to \eqref{eq:FrUF_approx}. Recall that
\begin{align}
\eps\le u_\eps \hbox{  in } \RN\times(0,+\infty),\qquad \label{eq:basso}\\
\int_{\RN}(u_\eps(x,t)-\eps)\d x = \int_{\RN}u_{0}(x)\d x,\,\,\,\,\forall t>0,
\label{eq:L1eps}\\
\hbox{ the sequence } u_\eps \hbox{ is monotone w.r.t. } \eps \label{eq:monotone}
\end{align}
From these properties we already know that have that the limit
function $ \bar u(x,t) := \lim_{\eps\searrow 0} u_\eps (x,t)$ is well defined and finite for any $(x,t) \in Q$.
Moreover, we have that the function $\bar u$
is such that $\bar u(\cdot,t)\in L^1(\RN)$ for
any $t>0$. More precisely, we have
\begin{equation}
\| \bar u(t)\|_{L^1(\RN)}\le \| u_0\|_{L^1(\RN)}\,\,\forall t>0.
\end{equation}
The main question is whether or not the mass conservation law is preserved in the limit,
since both possibilities are compatible with Fatou's theorem. We are going to prove that
actually, mass is not conserved and the limit function is just zero.

Let us take an arbitrary $\tau>0$.
We are going to prove that
\begin{equation}
\label{eq:conv_unif}
u_\eps \xrightarrow{\eps\searrow 0} 0\,\,\, \hbox{ uniformly on }
 \RN\times (\tau,+\infty).
\end{equation}
The uniform convergence is a consequence of
the fact that the initial condition $u_0$ is assumed to
be smooth and with compact support. We first need to prove the following lemma.

\begin{lemma}\label{Lem.Green} Under the running assumptions, the following equality holds true for any $\varepsilon>0$,  $0<\tau^*\le \tau$ and for any $x\in  \R^N$
\begin{equation}\label{eq:no_ex2}
\int_{\tau^*}^{\tau} \phi(u_\eps(\cdot,t))\d t - (\tau-\tau^*)\phi(\eps) =
 \int_{\RN}\frac{u_\eps(y,\tau)-u_\eps(y,\tau^*)}{\vert x-y\vert^{N-2s}} \d y\,.
\end{equation}
\end{lemma}
\begin{proof}Recall that $u_\varepsilon=v_\varepsilon +\varepsilon$, and that $v_\varepsilon$ satisfies equation \eqref{eq:FrUF_approx1}. We will work with $v_\varepsilon$, and we rewrite \eqref{eq:no_ex2} in terms of $v_\varepsilon$ in the form:
\begin{equation}\label{eq:no_ex3}
\int_{\tau^*}^{\tau} \phi_\varepsilon(v_\eps(\cdot,t))\d t =
 \int_{\RN}\frac{v_\eps(y,\tau)-v_\eps(y,\tau^*)}{\vert x-y\vert^{N-2s}} \d y\,.
\end{equation}
The proof of this important identity can be obtained in several ways. One way is to consider in the weak formulation of equation \eqref{eq:FrUF_approx1} the following test function
\[
\psi_\delta(t,x)=\chi_{[\tau^*,\tau]}(t)\,\Dsi\frac{\chi_{B_{\delta}(0)}(x)}{|B_{\delta}(0)|}\,.
\]
Notice that for almost every $x\in \R^N$,
\[
\Dsi\frac{\chi_{B_{\delta}(x)}(y)}{|B_{\delta}(x)|}\xrightarrow{\delta\searrow 0} \frac{C}{|x-y|^{N-2s}}\qquad\mbox{for almost every $y\in \R^N$}
\]
Unfortunately, this function is not an admissible test function. One possibility is to approximate $\psi_\delta$ by means of smooth functions, as it has been done for example in  Proposition 4.2 of \cite{BV2015}. Another possibility is to observe that $v_\varepsilon$ are smooth functions by construction, cf. \cite{VPQR13}, in particular they are classical solutions to equation \eqref{eq:FrUF_approx1} therefore we can multiply the equation by $\psi_\delta$ obtaining
\begin{equation}
\int_\R  \int_{\R^N} \psi_\delta\, \partial_t v_\eps(y,t)\,\d t \d y = - \int_\R  \int_{\R^N}  \psi_\delta \Ds \phi_\varepsilon(v_\varepsilon)\d t\ \d y
\end{equation}
Now we analyze the two terms separately. As for the left-hand side:
\begin{equation}\label{Lem.1}\begin{split}
&\int_\R  \int_{\R^N} \psi_\delta\, \partial_t v_\eps(y,t)\,\d t \d y=\int_\R \chi_{[\tau^*,\tau]}(t) \int_{\R^N} \partial_t v_\eps(\cdot,t) \Dsi\frac{\chi_{B_{\delta}(0)}(x)}{|B_{\delta}(0)|}\dt\d x\\
&=\int_{\RN}[v_\eps(y,\tau)-v_\eps(y,\tau^*)]\Dsi\frac{\chi_{B_{\delta}(x)}(y)}{|B_{\delta}(x)|}\d y
\xrightarrow{\delta\searrow 0}\int_{\RN}\frac{v_\eps(y,\tau)-v_\eps(y,\tau^*)}{\vert x-y\vert^{N-2s}} \d y\,.
\end{split}\end{equation}
the last limit can be justified by splitting the integral as follows.
First of all, we notice that
\begin{equation}
\label{eq:bound_sopra}
\Big\vert\Dsi\frac{\chi_{B_{\delta}(x)}(y)}{|B_{\delta}(x)|}\Big\vert\le C
\frac{1}{|x-y|^{N-2s}} \,\,\hbox{ in } \R^N,
\end{equation}
for a constant $C>0$ depending on $N$ and on $s$.
Then, we have that
\begin{equation}
\label{eq:bound_rhs}
\begin{split}
&\int_{\RN}\frac{v_\eps(y,\tau)-v_\eps(y,\tau^*)}{\vert x-y\vert^{N-2s}}\d y
=\int_{\RN} \frac{v_\eps(x-y,\tau)-v_\eps(x-y,\tau^*)}{\vert y\vert^{N-2s}}\d y\\
& =  \int_{B_{1}(0)}\frac{v_\eps(x-y,\tau)-v_\eps(x-y,\tau^*)}{\vert y\vert^{N-2s}}\d y
+ \int_{B^c_{1}(0)}\frac{v_\eps(x-y,\tau)-v_\eps(x-y,\tau^*)}{\vert y\vert^{N-2s}}\d y\\
&\le \|v_\eps(\cdot,\tau)-v_\eps(\cdot,\tau^*)\|_{L^\infty(\R^N)}\int_{B_{1}(0)}\frac{1}{\vert y\vert^{N-2s}}\d y
+ \int_{B^c_{1}(0)}|v_\eps(x-y,\tau)-v_\eps(x-y,\tau^*)|\d y\\
&\le 2(C\|v_\varepsilon(0)\|_{L^\infty(\R^N)}+\|v_\varepsilon(0)\|_{L^1(\R^N)})\le 2C\max\{\|u_0\|_{L^1(\R^N)}\,,\,\|u_0\|_{L^\infty(\R^N)}\}
\end{split}
\end{equation}
In the third line,  we have used that  $\vert y\vert^{2s-N}$ is integrable in $B_1(0)$ and $v_\eps(\cdot,\tau)-v_\eps(\cdot,\tau^*)$ is
bounded, while for $\vert y\vert>1$ the kernel $\vert y\vert^{2s-N}>1$ and $v_\eps(\cdot,\tau)-v_\eps(\cdot,\tau^*)$ is in $L^1(\RN)$. We conclude then by dominated convergence.

The right-hand side becomes:
\begin{equation}\begin{split}
&- \int_\R  \int_{\R^N}  \psi_\delta \Ds \phi_\varepsilon(v_\varepsilon)\d t\ \d y
= - \int_\R  \chi_{[\tau^*,\tau]}(t) \int_{\R^N}  \Dsi\frac{\chi_{B_{\delta}(x)}}{|B_{\delta}(x)|}
    \Ds \phi_\varepsilon(v_\varepsilon)\d t\ \d y\\
&= - \frac{1}{|B_{\delta}(x)|} \int_{B_{\delta}(x)} \int_{\tau^*}^\tau  \phi_\varepsilon(v_\varepsilon(y,t))\d t\ \d y
\xrightarrow{\delta\searrow 0}\int_{\tau^*}^\tau  \phi_\varepsilon(v_\varepsilon(x,t))\d t
\end{split}\end{equation}
where in the second line we have used Fubini's Theorem as follows:
\begin{equation}\label{Lem.2}\begin{split}
&\int_{\R^N}  \Dsi\frac{\chi_{B_{\delta}(x)}}{|B_{\delta}(x)|}
    \Ds \phi_\varepsilon(v_\varepsilon)  \d y
= \int_{\R^N} \left(\int_{\R^N} \frac{\chi_{B_{\delta}(x)}(z)}{|B_{\delta}(x)|}\frac{\d z}{|z-y|^{N-2s}}\right)
    \Ds \phi_\varepsilon(v_\varepsilon)  \d y  \\
&\int_{\R^N} \frac{\chi_{B_{\delta}(x)}(z)}{|B_{\delta}(x)|}\left( \int_{\R^N} \frac{\Ds \phi_\varepsilon(v_\varepsilon(y,t))  }{|z-y|^{N-2s}}\d y\right)\d z
=\int_{\R^N} \frac{\chi_{B_{\delta}(x)}}{|B_{\delta}(x)|} \Dsi \Ds \phi_\varepsilon(v_\varepsilon)  \d y
\end{split}\end{equation}
In the last step we have used the definition of Lebesgue point for $x \mapsto \int_{\tau^*}^\tau  \phi_\varepsilon(v_\varepsilon(x,t))\d t$.

Joining \eqref{Lem.1} and \eqref{Lem.2} we finally obtain \eqref{eq:no_ex3} and hence \eqref{eq:no_ex2} follows.
\end{proof}

\medskip

\noindent$\bullet~$\textsc{Step 2. }\textit{Convergence as $\varepsilon \searrow 0$ for bounded initial data. }Now may now proceed with the proof that  $u_\eps \xrightarrow{\eps\searrow 0} 0$\,, uniformly on $\RN\times (\tau,+\infty)$. To this end, let us fix $\tau$ and $\tau^*$ such that $0<\tau*<\tau$. The bound obtained in \eqref{eq:bound_rhs} combined with \eqref{eq:no_ex2}, gives
\begin{equation*}
\Big\vert \int_{\tau^*}^{\tau}\phi(u_\eps(x,t))\d t\Big\vert = - \int_{\tau^*}^{\tau}\phi(u_\eps(x,t))\d t\ge  -(\tau-\tau^*)\phi(\eps) - C,\,\,\,\,\forall x\in \RN,
\end{equation*}
where the constant $C$ depends only on $s$ and on $N$ and
is independent of $\eps$.
As a consequence, if we let $\eps\searrow 0$ and recall that
$\phi(\eps)\xrightarrow{\eps\searrow 0} -\infty$  (see \eqref{eq:nonlinearity}),
we have that
\begin{equation}
\label{eq:no_ex3}
\lim_{\eps\searrow 0}\int_{\tau^*}^{\tau}\phi(u_\eps(x,t))\d t = -\infty, \,\,\,\hbox{ uniformly for } x\in \RN.
\end{equation}
Next, we use time monotonicity. Since,
\begin{equation*}
\partial_t (t^{-1/(n+1)}v(t)) = t^{-1/(n+1)}\big(\partial_t v(t) - \frac{1}{(n+1)t}v(t) \big)\,\,\forall t>0,
\end{equation*}
the validity of the Aroson-B\'enilan estimate \eqref{eq:AB} is equivalent
to the decreasing monotonicity of $t^{-1/(n+1)}v(t)$ with respect to time.
Thus, we have that $u_\eps$ verifies
\begin{equation}
\label{eq:AB_decu}
u_\eps(x,t)\le \Big(\frac{t}{t_0}\Big)^{1/(n+1)}u_\eps(x,s), \hbox{ for any } x\in \RN \hbox{ and for any } t>t_0>0.
\end{equation}
On the other hand, using the definition of $\phi$ in \eqref{eq:nonlinearity} for $n>0$, monotonicity implies
\begin{equation}
\label{eq:AB_decw}
\phi(u_\eps(x,t)) \le \Big(\frac{t}{t_0} \Big)^{-n/(n+1)}  \phi(u_\eps(x,t_0)),\hbox{ for any } x\in \RN \hbox{ and for any } t>t_0>0.
\end{equation}
Therefore, for any $x\in \RN$ and $n>0$ we have
$$
\int_{\tau^*}^{\tau} \phi(u_\eps(x,t)) \d t \ge C(\tau, \tau^*,n) u_\eps(x,\tau)^{-n},
$$
where the constant $C$ can be computed explicitly.
As for for $n=0$, monotonicity implies
\begin{equation}
\label{eq:AB_decw_log}
\phi(u_\eps(x,t)) \le   \phi(u_\eps(x,s)) + \log\Big(\frac{t}{s}\Big),
\hbox{ for any } x\in \RN \hbox{ and for any } t>s>0,
\end{equation}
Therefore, for any $x\in \RN$ and $n=0$ we have
$$
\int_{\tau^*}^{\tau} \phi(u_\eps(x,t)) \d t \ge \log(u_\eps(x,\tau)) - C(\tau,\tau^*)
$$
where again the constant $C$ can be explicitly computed.

\medskip

\noindent$\bullet~$\textsc{Step 3. }\textit{Uniform convergence for bounded data. }All the above estimates imply that $\phi(u_\eps(\cdot,\tau))\xrightarrow{ \eps\searrow 0} -\infty$
uniformly on $\RN$, which implies that
$u_\eps(\cdot,\tau)\xrightarrow{\eps\searrow 0} 0$
uniformly in $\RN$. Thus, using the $L^\infty$-decay \eqref{eq:li_decay}, we can conclude \eqref{eq:conv_unif}.

\medskip

\noindent$\bullet~$\textsc{Step 4. }\textit{The case of unbounded initial conditions. }When dealing with initial conditions which are only integrable, we will obtain a weaker convergence than \eqref{eq:no_ex} (as in the statement of Theorem \ref{th:noex1}), by a standard density argument that we will discuss here. Given $u_0$ we prepare a sequence of bounded functions $0\le u_{0,M}\le M$, and with compact support, such that
$$
u_{0,M}\xrightarrow{M\nearrow +\infty} u_0 \,\,\,\,\hbox{ in } L^1(\RN).
$$
In particular, for any $k>0$, we find $\bar M$ such that
\begin{equation}
\label{eq:unb1}
\|u_{0,M}-u_0\|_{L^1(\RN)}\le \frac{1}{k}, \,\,\,\forall M>\bar M.
\end{equation}
Let $u_{\eps,M}$ be the solution to \eqref{eq:FrUF_approx}
starting from $u_{0,M}+\eps$. and let $M>\bar M$.
Since $u_{0,M}\in L^\infty(\RN)$, we know from the
analysis above that in the appropriate range of  $N,n,s$, we have
\begin{equation*}
\lim_{\eps\searrow 0} u_{\eps,M} = 0,
\end{equation*}
uniformly in $\RN\times (\tau,+\infty)$, $\forall \tau>0$.
Fix now $\tau>0$ and $k>0$, and let $M$ as in \eqref{eq:unb1}.
 Then, there exists $\eps_0$ such that for $\eps<\eps_0$, there holds
$$
u_{\eps,M}(x,t)\le \frac{1}{k} \,\,\,\hbox{ in } \RN\times (\tau,+\infty).
$$
On the other hand, using \eqref{eq:unb1} and \eqref{eq:contraction},
we have for $M>\bar M$
$$
\| u_{\eps,M}(\cdot,t) - u_{\eps}(\cdot,t)\|_{L^1(\RN)}\le \frac{1}{k}.
$$
Thus, for every ball $B$ in $\RN$ we conclude that
$$
\lim_{\eps\searrow 0}u_{\eps} = 0 \hbox{ in } L^1(B) \,\,\,\hbox{ uniformly in }t\ge \tau \hbox{ for any }\tau>0\,.
$$
This concludes the proof of Proposition \ref{th:noex1}, hence of Theorem \ref{th:non ex} and of Theorem \ref{th:non_ex_1D} only when $N>2s$.

\section{Trivial limit solution in dimension 1}\label{sec.1D}

We now deal with the one-dimensional case, $N=1$. This section contains the proof of Theorem \ref{th:non_ex_1D} in the cases not included in the previous section ($2s<1$), namely we consider here the remaining two cases: $s\in (1/2,1)$ with $n\ge 2s-1$ and $s=1/2$ with $n>0$.

Note also that these conditions are in complete agreement with the existing results available for the standard Laplacian, i.e. $s=1$, where existence of nontrivial solutions holds when $N=1$ and $n<1$. Here we prove non-existence in the remaining cases.

These two cases require a quite different technical treatments, mainly because the Green function has really a different form from the other cases, indeed it grows at infinity.

The proof of the optimality of the range of exponents for this one dimensional case will be given in Section \ref{sec.exist.1D}.\normalcolor

\subsection{Non existence for $s\in (1/2,1)$}

We now analyze the trivial limit in the situation $N=1$, $s\in (1/2,1)$ (so that the rule $N>2s$ does not apply) and $n\ge 2s-1$. Thanks to the monotonicity \eqref{eq:monotone}, the limit
\[
\bar u := \lim_{\eps\searrow 0} u_\eps
\]
is well defined and finite for any $(x,t) \in Q = \R\times (0,+\infty)$
and is such that $\bar u(\cdot,t)\in L^1(\R)$ for any $t>0$. We need a preliminary lemma.

\begin{lemma}\label{Lem.Green.1}
Under the running assumptions, the following equality holds true for any $\varepsilon>0$,  $0<\tau\le \bar\tau$ and for any $x\in  \R$
\begin{equation}
\label{eq:no_ex2_1D.000}
\int_{\tau}^{\bar t} \phi(u_{\eps}(x,t))\d t = \int_{\tau}^{\bar t} \phi(u_{\eps}(0,t))\d t
+\int_{\R}\rho_\eps(-y)\big(\vert x + y\vert^{2s -1} - \vert y\vert^{2s -1} \big) \d y.
\end{equation}
where \ $  \rho_\eps(x) := u_{\eps}(x,\bar t) - u_\eps(x,\tau)$,  for $ x\in \R. $
\end{lemma}

\begin{proof}
For any $\delta>0$, let $\zeta_\delta$ denote a smooth function with compact support converging to the delta function in the origin as $\delta\searrow 0$, namely
\[
\lim_{\delta\searrow 0}\int_{\R}\zeta_\delta(x) \phi(x) \d x = \phi(0), \,\,\,\forall \phi \hbox{ smooth. }
\]
We can assume that $\zeta_\delta\ge 0$ and that $\int_{\R}\zeta_\delta(x) d x = 1.$
Then, denote with $G_\delta$ the function given
by expression
\begin{equation}
\label{eq:Gdelta}
G_\delta(x):=\int_{\R}\vert x- y\vert^{2s -1}\zeta_\delta(y)\d y.
\end{equation}
Note that,
even if $2s\ge 1$, the function $G_\delta$ is well defined
and finite for any $x$ in $\R$
due to the fact that $\zeta_\delta$ has compact support.
As a matter of fact, we have
that $\Ds G_\delta = \zeta_\delta$.
Moreover, since $\int_{\R}\zeta_\delta \d x=1$,
using the elementary inequality
\begin{equation}
\label{eq:ine}
\vert  A^\alpha -B^\alpha\vert\le \vert  A - B \vert^{\alpha},
\,\,\,\alpha\in (0,1), \,\,\forall A, B\in (0,+\infty),
\end{equation}
 we have
\begin{equation}
\label{eq:uniformv}
\vert G_\delta(x) - \vert x \vert^{2s-1}\vert =
\vert \int_{\R}\big(\vert x-y\vert^{2s-1} - \vert x\vert^{2s-1}\big)\zeta_\delta( y) \d y\vert
\le \int_{\R}\vert y\vert^{2s-1}\zeta_\delta(y) \d y \,\,\,\,\forall x\in \R.
\end{equation}
Then, if we let $\delta\searrow 0$, we immediately get that
$G_\delta(x)\xrightarrow{\delta\searrow 0}\vert x\vert^{2s-1}$ uniformly in $\R$.
Now, for $x\in \R$,
we multiply the equation \eqref{eq:FrUF_approx} with $G_\delta(x-\cdot) -G_\delta(\cdot)$.
and we integrate on $\R\times(\tau, \bar t)$, with $0<\tau<\bar t$. We get,
\begin{equation}
\label{eq:no_ex1_1D}
\int_{\R}\rho_\eps(y) \Big(G_\delta(x-y) - G_\delta(y) \Big)\d y + \int_{\R}\big(W_\eps(y)
-(\bar t-\tau)\phi(\eps)\big)(\zeta_\delta(x-y) - \zeta_\delta(y))\d y = 0,
\end{equation}
where
\begin{equation}\label{notations.W}
W_\eps(x) := \int_{\tau}^{\bar t} \phi(u_{\eps}(x,t))\d t, \hbox{ for } x\in \R.
\end{equation}
Note that since $n\ge 2s-1>0$, we have
$\phi(u_\eps) = -u_{\eps}^{-n}$.
Now, using \eqref{eq:uniformv} we have that,
for any fixed $x\in \R$ and for a sufficiently
small $\delta$,
\begin{align*}
\vert \rho_\eps(y) \Big(G_\delta(x-y) - G_\delta(y) \Big)\vert &\le
\vert \rho_\eps(y)\vert +
 \vert \rho_\eps(y)\vert\Big(\vert x-y\vert^{2s-1}-\vert y\vert^{2s-1}\Big)\vert\\
&\le (1+ \vert x\vert^{2s-1} )
\vert \rho_\eps(y)\vert, \,\,\,\forall y\in \R
\end{align*}
where we used \eqref{eq:ine} in the last inequality.
On the other hand, we have that for any $x\in \R$
 the function $y\mapsto (1 +\vert x\vert^{2s-1})\vert \rho_\eps(y)\vert$
 is in $L^1(\R)$.
Hence, recalling that (for fixed $\eps$) $W_\eps$ is a smooth function, we can pass to the limit as $\delta\searrow 0$ in a standard way
and obtain
\begin{equation}
\label{eq:no_ex2_1D}
W_\eps(x) -W_\eps(0) =
\int_{\R}\rho_\eps(-y)\big(\vert x + y\vert^{2s -1} - \vert y\vert^{2s -1} \big) \d y\,,
\end{equation}
which is exactly \eqref{eq:no_ex2_1D.000}.\end{proof}

We are now in the position to prove that $\bar u := \lim_{\eps\searrow 0} u_\eps \equiv 0$. We reason by contradiction and we assume that $\bar u(0,\bar t)>0$ for some $\bar t>0$. \normalcolor We keep the notations \eqref{notations.W} for $W_\eps$ and $\rho_\varepsilon$.

Combining \eqref{eq:no_ex2_1D.000} with \eqref{eq:ine} (in the equivalent form \eqref{eq:no_ex2_1D}, we obtain that $W_\eps$ is H\"older continuous, namely
\begin{equation}
\label{eq:holder}
\vert W_\eps(x) -W_\eps(0)\vert \le C\vert x\vert^{2s -1} \hbox{ for } x\in \R,
\end{equation}
with
\[
\sup_{\varepsilon>0}\int_{\R} \vert \rho_\eps\vert \d x\le C<+\infty,
\]
which gives
\[
  W_\eps(x) \ge W_\eps(0)-C\vert x\vert^{2s-1}.
\]
Thus, using the  Benil\'an-Crandall estimate in the form
\eqref{eq:AB_decw}, we have that there exists a constant $C=C(\tau,\bar t,n,\|u_0\|_{L^1(\R)})$
such that
\begin{equation}
\phi(u_\eps(x,\tau))\ge C\Big(\phi(u_\eps(0,\bar t)) -  \vert x\vert^{2s-1} \Big) \hbox{ for any } x\in \R.
\end{equation}
Recaling that $\phi(u)=-u^{-n}$, we have that
\begin{equation}
\label{eq:le}
u_\eps(x,\tau) \ge \left( \frac{C}{\frac{1}{u_\eps(0,\bar t)^n}  + \vert x\vert^{2s-1}}   \right)^{1/n}.
\end{equation}
Thus, if we let $\eps\searrow 0$, the inequality above is conserved in the limit, namely
\begin{equation}
\bar u(x,\tau) \ge \Big( \frac{C}{\frac{1}{\bar u(0,\bar t)^n}  + \vert x\vert^{2s-1}}   \Big)^{1/n}.
\end{equation}
Now, recall that we are assuming that $\bar u(0,\bar t)>0$ (hence $1/\bar u(0,\bar t)$ is finite).
This is manifestly
impossible if $n\ge 2s-1$ since the function $(1+ \vert x\vert)^{-(2s -1)/n}$ is integrable
only if $n<2s-1$, while the function $\bar u(\cdot,t)$ is in $L^1(\R)$ for any $t>0$.

This concludes the proof of Theorem \ref{th:non_ex_1D} when $s\in (1/2,1)$ with $n\ge 2s-1$.

\subsection{Non existence when $s=1/2$ and $n>0$}\label{subsec.1D.mezzo}

Now we are going to discuss the case $s=1/2$ and $n>0$. Recall that for $n=0$ there is an explicit positive solution, given in \eqref{explicit.sol.mezzo}. This situation presents some extra difficulties because
for $s=1/2$ the Green function is of logarithmic type and the corresponding
integral analogous to the one in the right hand side of
\eqref{eq:no_ex2_1D} seems difficult to handle. We rely on the analysis
for $s>\frac{1}{2}$ by noting that the lower bound \eqref{eq:le}, which is one of the
main points in the argument above, has a ``good dependence''
on $s$, namely the constant $C$ in \eqref{eq:le} does not depend on $s$.
Thus, our strategy in proving a lower bound like \eqref{eq:le}
for the solution of the $s=1/2$ case consists in approximating this solution
with solutions of the problem \eqref{eq:FrUF_approx} with fractional laplacian of order $s>1/2$.
To make this argument rigorous,
we need the following lemma on the continuity of the solution
operator with respect to the fractional order of derivation
\begin{lemma}
\label{lemma:onehalf}
Let $\psi:\R \to \R$ be a smooth and non decreasing function with $\psi(0)=0$.
Let $s_\delta$ be a sequence of real numbers with $s_\delta>\frac{1}{2}$
and such that $s_\delta\xrightarrow{\delta\searrow 0} 1/2$.
Let $v_\delta$ be the solution to
\begin{equation}
\label{eq:probdelta}
\begin{cases}
\partial_t v_\delta + \Dsd \psi(v_\delta) = 0 \,\,\,\hbox{ in } Q\\
v_\delta(0) = v_{0,\delta} \,\,\,\hbox{ in } \R.
\end{cases}
\end{equation}
Then, if $v_{0,\delta}\xrightarrow{\delta\searrow 0} v_0$
in $L^1(\R)$, there holds that
$v_\delta\xrightarrow{\delta\searrow 0} v$ in $C^0([0,+\infty);L^1(\R))$
and $v$ is the solution to
\begin{equation}
\label{eq:probnehalf}
\begin{cases}
\partial_t v + \Doh \psi(v) = 0 \,\,\,\hbox{ in } Q\\
v(0) = v_{0} \,\,\,\hbox{ in } \R.
\end{cases}
\end{equation}
\end{lemma}
\begin{proof} The proof is an adaptation of the one given in \cite[Theorem 10.1, Section 10]{pqrv2}.
\end{proof}\normalcolor
We use the above Lemma in the following way. For any $\eps>0$ denote with $u_\eps$ the solution to
\begin{equation*}
\begin{cases}
\partial_t u_\eps + \Doh (\phi(u_\eps)) = 0, \,\,\,\hbox{ in } \R\times (0,+\infty)\\
u_\eps(0) = u_{0,\eps}=u_0 + \eps \,\,\,\hbox{ in } \R.
\end{cases}
\end{equation*}
The above Lemma shows that $u_\eps$ can be approximated when
$\delta\searrow 0$ with the solution $u_{\delta,\eps}$  of
\begin{equation*}
\begin{cases}
\partial_t u_{\delta,\eps} + \Dsd(\phi(u_{\delta,\eps})) =0,
 \,\,\,\hbox{ in } \R\times (0,+\infty) \\
u_{\delta,\eps}(0) = u_{0,\delta,\eps} \,\,\,\,\hbox{ in } \R.
\end{cases}
\end{equation*}
To be precise, the above Lemma shows that the solution $v_\eps=u_\eps-\eps$ can be approximated by $v_{\delta,\eps}:=u_{\delta,\eps}-\eps$.

We are in the position to use Lemma \ref{Lem.Green.1} and repeat all the subsequent argument, for the approximations $u_{\delta,\eps}$ which fall in that range of parameters. Denote, for any $(x,t) \in \R\times (0,+\infty)$,
 $\bar u(x,t) :=\lim_{\eps\searrow 0}u_\eps(x,t)$
and assume by contradiction that $\bar u(0,\bar t)>0$. Since $s_\delta\searrow 1/2$
as $\delta\to 0$, we can fix $\delta$ small enough in such a way that $n\ge 2s_\delta-1>0$.
Thus, we have that
$u_{\delta,\eps}$ satisfies the lower estimate \eqref{eq:le} which becomes
\begin{equation*}
u_{\delta,\eps}(x,\tau) \ge \Big( \frac{C}{\frac{1}{u_{\delta,\eps}(0,\bar t)^n}
+ \vert x\vert^{2s_\delta-1}}   \Big)^{1/n}
\end{equation*}
for some $\tau<\bar t$. It is important to notice that the constant $C$ that appears
above, as well as the other constants that will come up in what follows, does not depend
on $s$ and hence on $\delta$. Thus, if we let $\delta\searrow 0$, we have a similar
lower bound for $u_\eps$, namely
$$
u_{\eps}(x,\tau) \ge \Big( \frac{C}{\frac{1}{u_{\eps}(0,\bar t)^n}  + 1}   \Big)^{1/n},
\,\,\,\,\hbox{ for any } x\in \R.
$$
Now, we can let $\eps\searrow 0$, obtaining that
$$
\bar u(x,\tau) \ge \Big( \frac{C}{\frac{1}{\bar u(0,\bar t)^n}  + 1}   \Big)^{1/n},
\,\,\,\,\hbox{ for any } x\in \R,
$$
which contradicts the integrability of $\bar u(\cdot,\tau)$.
Thus, we have that $\bar u\equiv 0$ in $\R\times (0+\infty)$.

This concludes the proof of Theorem \ref{th:non_ex_1D} when $s=1/2$ and $n>0$.


\section{Existence of nontrivial solutions in 1D}\label{sec.exist.1D}

Let us now examine the situation for $N=1$ in the remaining cases, namely $2s-1>n\ge 0$, and $s=1/2$, $n= 0$, where we expect to find nontrivial  limit solutions, as announced in the statement of Theorem \ref{th:non_ex_1D}. This will prove the optimality of the range of exponents considered for non-existence and conclude the proof of Theorem \ref{th:non_ex_1D}.

\subsection{Case $s=1/2$, $n=0$}

This is a kind of exceptional case in the parameter diagram, see Figure 1. Indeed,  for $s=1/2$ there exists an explicit solution  of the evolution problem
\begin{equation}
\label{eq:logonehalf}
\begin{cases}
\partial_t u + \Doh(\log u) = 0\,\,\,\,\hbox{ in } \R\times (0,+\infty),\\
u(0) = u_0 \,\,\,\,\hbox{ in } \R,
\end{cases}
\end{equation}
with an initial
condition in $L^1(\R)$. It is given by the formula
\begin{equation}\label{explicit.sol.mezzo}
U(x,t)=\frac{2(T-t)}{1+|x|^2} \,\,\,\hbox{ in } \R\times (0,+\infty).
\end{equation}
We see that it has the separate-variable type as in the previous subsection, and  $U(\cdot, t)$ is in $L^1(\R)$ for any $t\ge 0$. It is very peculiar that the solution becomes identically zero in finite time. This is the so-called {\sl finite time extinction} phenomenon which is typical
of some ranges of fast diffusion, see \cite{JLVSmoothing} for standard diffusion and \cite{pqrv2, KL} for fractional diffusion. A simple comparison theorem implies that any  initial data $u_0(x)\ge c/(1+|x|^{2})$
produces a nontrivial limit solution equal or larger than $U(x,t)$ with $T\le c/2$, hence non trivial in a time interval $0<t<T$. In conclusion, by proving the existence of nontrivial solutions when $n=0$, we have made sure that in one dimension for $s=1/2$ the condition $n>0$ is sharp for non-existence.

\subsection{Existence for $s>1/2$, $n<2s-1$ }

Even if the diffusion nonlinearity is singular, the range is formally the same as the ``good fast diffusion'' range $1>m> (N-2s)/N$, considered in the general theory of \cite{pqrv2}, hence it is supercritical in the notation of that paper. But we recall that there only exponents $m>0$ were considered. The presence of the singular nonlinearity when $m\le 0$ makes it impossible to use just the same arguments to develop a general existence theory, though a number of basic results are expected to be the same.

We refer to the recent paper \cite{V-exist} by the third author, in which the Barenblatt solutions are constructed in this range of parameters; the proofs starts by the existence of a positive sub-solution. The sharp behaviour at infinity of Barenblatt solutions is established.

Finally, we mention that there exist a special class of {\sl very singular solutions} which are solutions of equation $u_t+(-\Delta)^s (u^m/m)=0$ in an appropriate sense specified in \cite{V-exist}, and has the explicit form
\begin{equation}
{}U(x,t)=C\,(T-1)^{1/(1-m)}|x|^{-2s/(1-m)},
\end{equation}
with a constant $C=C(m,N,s)>0$ which can be explicitly determined, see \cite{V-exist}. It happens that the spatial profile has a non-integrable singularity at $x=0$. It is also proved that such a formal solution is the limit of a monotone increasing sequence of Barenblatt solutions, we refer to \cite{V-exist} for further details.

The proof of Theorem \ref{th:non_ex_1D} is now complete.

\normalcolor


\section{Non-existence for classes of generalized solutions}
\label{sec.nonweak}

The trivial limits obtained in Theorems \ref{th:non ex} and \ref{th:non_ex_1D} are a clear indication that no solutions must exist in any other reasonable sense, since the approximation is monotone from above and the maximum principle is known to hold for regular solutions; hence we expect the non-existence for other classes of solutions. Indeed, this last fact needs proof, since limit of solutions could turn out to be unreliable.

Loosely speaking, we seek  a class of solutions for which the comparison principle holds. It turns out that such class is  difficult to choose, various non-equivalent classes of solutions obey the maximum principle, for example strong or {mild} solutions. We will prove the needed comparison result for strong solutions, because we are interested in short and clear arguments. It shall be mentioned that bounded (energy or even very) weak solutions are expected to be strong solutions, this is indeed true when $m>0$, cf. \cite{pqrv2, ottoL1}, but when $m\le 0$, this fact requires a different proof, which is quite long and technical and falls out the scope of this paper.

\begin{theor} \label{th:comparison}
Let $u$ be a strong solution of Equation $\eqref{geq.Phi}$ in the sense of Definition $\ref{def:weak_sol}$, corresponding to the initial datum $u_0\in L^1_+(\RN)$ and let $u_\eps$ be the solution of the Cauchy problem $\ref{eq:FrUF_approx}$\,, with initial datum $u_{\eps,0}:=u_0 +\eps$. Then,
\begin{equation}
\label{eq:comparison}
u \le u_\eps \,\,\,\,\hbox{ a.e. in }\, Q=  \RN\times (0,+\infty)\,.
\end{equation}
\end{theor}
\begin{proof}Let $u$ be a strong solution, hence $\partial_t u\in L^1(Q)$\,. We denote with $p_\delta$ a smooth approximation of the sign function, namely $p_\delta$ is a a $C^1$ function with $0<p_\delta<1$, $p_\delta\equiv 0$ for $y\le 0$ and $p'_\delta>0$ for $y>0$. We take the difference between the equation for $u$ and the equation for $u_\eps$ and we test it with $p_\delta(\phi(u)-\phi(u_\eps))$. We have
\[
\int_{\RN}\partial_t(u-u_\eps)p_\delta(\phi(u)-\phi(u_\eps))\d x=-
\int_{\RN}\big[p_\delta(\phi(u)-\phi(u_\eps))\big]\,\Ds(\phi(u)-\phi(u_\eps)) \d x\,.
\]
Now, the second integral is non-negative thanks to the Stroock-Varopoulos
inequality, that reads in this case:
\begin{equation}\label{Strook-Varopoulos}
\int_{\RN}(p_\delta(\phi(u)-\phi(u_\eps)))\Ds(\phi(u)-\phi(u_\eps)) \d x
\ge \int_{\RN}\left|\Dsm(\Psi_\delta(\phi(u)-\phi(u_\eps)))\right|^2 \d x\ge 0
\end{equation}
where $p'_\delta=(\Psi'_\delta)^2$. We refer to \cite{pqrv2}, Lemma 5.2 for a proof.
Indeed, the above proof is formal, since the above inequality holds a class of functions with fast decay at infinity. To make it rigorous one should consider further approximations $\phi(u_j)-\phi(u_{\eps,j})$ with a fast decay at infinity, for example in the Schwartz class $\mathcal{S}(\RN)$ and then pass to the limit, noticing that the sign is preserved under such limit process. We leave the details to the interested reader. \normalcolor Thus, we have
\begin{equation}
\label{eq:comparison1}
\int_{\RN}\partial_t(u(x,t)-u_\eps(x,t))p_\delta(\phi(u(x,t))-\phi(u_\eps(x,t)))\d x\le 0 \,\,\,\,\hbox{ for any } t>0.
\end{equation}
Since $u_\eps = v_\eps+ \eps$ (by definition) and $v_\eps$ is a strong solution,
we have that
$\partial_t u_\eps\in L^1(\RN)$ for any $t>0$. Thus, we can use the dominated
convergence Theorem to remove the $\delta$ approximation in \eqref{eq:comparison1}.
Thus, we get
$$
\frac{\rm d}{{\rm d}t}\int_{\RN}(u(x,t)-u_\eps(x,t))_+\d x\le 0, \,\,\,\,\,\hbox{ for any } t>0,
$$
where we have used that $\hbox{sign}_{+}(u-u_\eps)=\hbox{sign}_+(\phi(u)-\phi(u_\eps))$ thanks
to the monotonicity of $\phi$. Thus, integrating
the above relation between $0$ and some arbitrary $t>0$, we get
$$ \int_{\RN}(u(x,t)-u_\eps(x,t))_{+}\d x\le \int_{\RN}(-\eps)_{+}\d x = 0,$$
that implies \eqref{eq:comparison}.
\end{proof}

Clearly, once we have this comparison principle at our disposal
the nonexistence result applies  trivially to strong solutions, thus providing the proof of Theorem \ref{th:nonex-strong}.

\section{Comparison between different nonlinearities}\label{sec.comparison.nonlin}

Up to now we have considered the problems of existence or non-existence of nontrivial limit solutions for equations of the type \eqref{geq.Phi} in the case of the most typical nonlinearities, those of the form $\phi(u)=u^m$ with $m>0$,  $\phi(u)=-u^{-n}$ for $n>0$, of $\phi(u)=\log(u)$ (which is a kind of case $m=n=0$. The results obtained in our previous sections can be extended to more general nonlinearities $\phi$ using the symmetrization comparison results proved by one of the authors and Volzone in papers \cite{VazVol1}, \cite{VazVol2}. The first one contains the basic comparison
results for solutions of the same equation, as well as the needed concepts and tools.
The comparison of solutions of equations with different nonlinearities is presented in \cite{VazVol1}: first, we need to recall the concepts of symmetric rearrangement of a positive real function and the order relation $\prec$ that can be taken from Section 1 of \cite{VazVol1}. Then we need the concept of comparison of nonlinearities, also called diffusivities (a bit vaguely).

\begin{defn}
Assume that $\phi,\Phi:\R_{+}\rightarrow\R{+}$ are two functions which are smooth in $(0,\infty)$. We say $\Phi$ is slower than $\phi$ if
\[
\Phi^{\prime}(r)\leq \phi^{\prime}(r), \quad\forall r>0.
\]
In this case we also say that $\Phi$ is less diffusive than $\phi$.
\end{defn}

Once we have introduced the concept of {\itshape diffusivities}, the basic parabolic
comparison for solutions of the same equation with different diffusivities is
proved in \cite{VazVol2}

\begin{theor}[Teorem 3.2 \cite{VazVol2}]\label{thm.comp.diff.diff}
Let $u$ be the nonnegative mild solution to problem
\begin{equation*} \label{eqcauchy.f}
\left\{ \begin{array} [c]{lll}%
u_t+(-\Delta)^{s}\phi(u)=f  &  & x\in\R^{N}, \ t>0\,,%
\\[6pt]
u(x,0)=u_{0}(x) &  & x\in\R^{N},
\end{array}
\right. %
\end{equation*}
 with $0<s<1$, with initial data $u_0\in L^1(\RN)$, $u_0\ge 0$, right-hand side $f\in L^1(Q)$, $f\ge 0$, and smooth, strictly increasing nonlinearity $\phi(u)$ on $\R_+$ with
$\phi(0)=0$. Assume that $\Phi$ is a concave nonlinearity on $\R_+$ satisfying the same assumptions of $\phi$ and let $v$ be the mild solution to
\begin{equation}\label{symmorconc}
\left\{
\begin{array}
[c]{lll}%
v_t+(-\Delta)^{s}\Phi(v)=\widetilde{f}(|x|,t)  &  & x\in\R^{N}\,, \ t>0,%
\\[6pt]
v(x,0)=\widetilde{u}_{0}(x) &  & x\in\R^{N},
\end{array}
\right. %
\end{equation}
where $\widetilde{f}\in L^{1}(Q)$, $\widetilde{u}_{0} \in L^{1}(\R^{N})$ are nonnegative, radially symmetric decreasing functions with respect to $x$. If moreover $\Phi$ is slower than $\phi$, and
\[
u_{0}^{\#}(|x|)\prec\widetilde{u}_{0}(|x|),\quad f^{\#}(|x|,t)\prec\widetilde{f}(|x|,t),
\]
for almost all $t>0$, then  the conclusion $u^\#(|x|,t)\prec v(|x|,t)$    holds.
\end{theor}

We apply the Theorem above to a pair of nonlinearities $\phi(u)$ and $\Phi_n(u)$, where $\Phi_n(u)=-u^{-n}$ for $n>0$ and $\Phi_0(u)=\log(u)$. We choose the case where $n$ satisfies the conditions of the non-existence theorems, and we get the following improved result.

\begin{theor}\label{Thm.non_ex-Symmetriz} The nonexistence results of  Theorem $\ref{th:non ex}$ and Theorem $\ref{th:non_ex_1D}$
still hold when instead of taking $\phi=\Phi_n$ we assume that $\phi$ is smooth, strictly increasing nonlinearity $\phi(u)$ on $\R_+$ with $\phi(0)=0$, and moreover
\begin{equation}
\phi'(u)\ge C(M)\, u^{-(n+1)}\qquad \mbox{for all } \ 0<u<M,
\end{equation}
and all $M>0$. The range of accepted $(N,s,n)$ is as in Theorem $\ref{th:non ex}$ and Theorem $\ref{th:non_ex_1D}$.
\end{theor}

\begin{cor} If a non-existence result is proved for a choice of $(N,s,n)$, it also hold for
 $(N,s,n_1)$ with $n_1>n$.
\end{cor}

This corollary sheds light into the conditions  of the phenomenon of non-existence.



\section{The elliptic problem}\label{sec.elliptic}

The techniques used to study the parabolic problem  allow us to prove analogous non-existence results for the following elliptic equation
\begin{equation}
\label{eq:elliptic.intro}
u + \Ds \phi(u) = f \,\,\,\,\hbox{ in } \RN,
\end{equation}
when $f\in L^1(\RN)$ and the nonlinearity $\phi$ is given by the formulas \eqref{eq:nonlinearity}. Writing $\phi(u)=v$ hence $u=\phi^{-1}(v):=\beta(v)$, we can rewrite it in the form
\begin{equation}
\label{eq:elliptic.intro2}
\Ds v +\beta(v)= f \,\,\,\,\hbox{ in } \RN.
\end{equation}
This is the fractional version of the equation treated by B\'enilan, Brezis and Crandall, for $s=1$ in their famous paper \cite{BBC}\,.
In the paper, they study existence for equation \eqref{eq:elliptic.intro2}
assuming that $\beta$ is a monotone function. More precisely, they assume
that $\beta$ is a maximal monotone graph (hence possibly multivalued) with
the property that $0\in \beta(0)$. On the other hand, our nonlinearity satisfies
$\beta(-\infty)=0$, which is the cause for the non-existence result that is not considered by them.

Equation \eqref{eq:elliptic.intro} appears as the iteration step on solving the parabolic equation \eqref{geq.Phi} by the implicit time discretization scheme suggested by the Crandall-Ligget Theorem \cite{CL71}. This explains why the solutions we get for the elliptic problem, Theorems \ref{th:elliptic_1} and \ref{th:non_ex_elliptic}, are so similar to the parabolic ones.

\subsection{Nonexistence of limit solutions}
We consider an approximate problems with $f\in L^1(\RN)\cap L^\infty(\RN)$
\begin{equation}
\label{eq:elliptic_approx1}
\begin{cases}
v + \Ds( \phi_\eps(v)) = f \hbox{ in } \RN\\
\phi_\eps(v):= \phi(v+\eps) - \phi(\eps).
\end{cases}
\end{equation}
Notice that for any $\eps>0$ the nonlinearity $\phi_\eps$
is smooth and monotone increasing. Therefore, the above
problem can be solved in a standard way
(see the discussion in \cite[Section 8.1]{VPQR13}),  and the solution turns out to be unique, bounded and nonnegative. Moreover,
\begin{equation}
\label{eq:linftyelliptic}
\| v_\eps\|_{L^\infty(\RN)}\le \|f\|_{L^\infty(\RN)}.
\end{equation}
and the standard comparison theorem holds.
Let us denote by $v_\varepsilon$ such solutions.
Next we define $u_\eps:= v_\eps + \eps$, noticing that $u_\varepsilon$ is the (unique) solution to
\begin{equation}
\label{eq:elliptic_approx2}
u_\eps + \Ds \phi(u_\eps) = f_\eps:=  f + \eps \,\,\hbox{ in } \RN\,.
\end{equation}
We also have
\begin{align}
& u_\eps\ge \eps \,\,\,\,\,\,\,\hbox{ in }\RN,\label{eq:basso}\\
&\| u_\eps - \eps\|_{L^1(\RN)}\le  \| f\|_{L^1(\RN)}\label{eq:L1u}
\end{align}
Recall that the sequence $u_\eps$ is monotone with respect to
$\eps$, therefore the limit function
\begin{equation}\label{limit-elliptic}
 \bar u(x) :=\lim_{\eps\searrow 0}u_\eps(x)
 \end{equation}
exists for almost any $x\in \RN$. Moreover, we have that
$0\le \bar u(x) \le \|f\|_{L^\infty(\RN)}$. In the case where $f\in L^1(\RN)$ we can use the property of $L^1$ contraction to show that there is also a limit for formula \eqref{limit-elliptic} though the function $u_\varepsilon$ are not necessarily bounded.

Analogously to the parabolic case, we will prove the following result.
\begin{theor}\label{th:elliptic_1}
For any given $f$ with  $f\in L^1(\RN)$, let $u_\eps$ be the solution of \eqref{eq:elliptic_approx2}.
Then the trivial limit
\begin{equation}
\label{eq:zero_lim_elliptic}
\lim_{\eps\searrow 0} u_\eps = 0 \hbox{ in } L^1_{loc}(\RN)
\end{equation}
holds true whenever $N\ge 2$ and $n\ge 0$, or $N=1$ and
\[
\mbox{$n\ge 0$ and $s\in (0,1/2)$; or $n> 2s-1$ and $s\in (1/2,1)$; or $n>0$ and $s=1/2$.}
\]
This range of exponents is optimal.
\end{theor}
\begin{proof}The proof will be split into several steps. We repeat the scheme of the proof for the parabolic case, therefore we will just sketch the main arguments. We will begin with $f\in L^1(\RN)\cap L^\infty(\RN)$ and in Step 5 a density argument will allow us to treat the case $f\in L^1(\RN)$. The last step is devoted to prove the sharpness of the exponent's range.

\noindent$\bullet~$\textsc{Step 1. }\textit{The case $N> 2s$. }This case includes the proof of both the case  $N\ge 2$ and  $N=1$, $s\in(0,1/2)$. Proceeding as in Lemma \ref{Lem.Green}, multiplying by the test function
\[
\psi_n(\cdot)=\Dsi\left(\frac{\chi_{B_{1/n}(x)}}{|B_{1/n}(x)|}\right)(\cdot)
\]
and passing to the limit $n\to\infty$ we obtain
\begin{equation}
\phi(u_\varepsilon(x)) -\phi(\eps) = \int_{\RN}\frac{f(x) - u_\varepsilon(x)+\varepsilon}{\vert x- y\vert^{N-2s} }\d y.
\end{equation}
Notice that the right-hand side is uniformly bounded with respect to $\varepsilon>0$ and $x\in\RN$, because
\begin{equation}
\label{eq:bound_rhs.ELL}
\begin{split}
\int_{\RN}\frac{f(x) - u_\varepsilon(x)+\varepsilon}{\vert x- y\vert^{N-2s} }\d y
\le 2C\max\{\|f\|_{L^1(\R^N)}\,,\,\|f\|_{L^\infty(\R^N)}\}
\end{split}
\end{equation}
the proof being identical to the parabolic case, see formula \eqref{eq:bound_rhs} in the proof of Lemma \ref{Lem.Green}.
As a consequence, the following limits are uniform in $x\in \RN$:
$$
\lim_{\eps\searrow 0}\phi(u_\varepsilon(x)) = -\infty\,,\qquad \lim_{\eps\searrow 0} u_\varepsilon(x)  = 0\,.
$$
This clearly implies that $\bar u\equiv 0$ in $\RN$.\\

\noindent$\bullet~$\textsc{Step 2. }This case corresponds to $N=1$, $s\in (1/2,1)$, and $n\ge 2s-1$. Let
\[
\overline{u}(x):=\lim_{\eps\searrow 0} u_\varepsilon(x)\,,
\]
and assume by contradiction that $\overline{u}(0)>0$.
Proceeding as in Lemma \ref{Lem.Green.1}, we obtain the inequality
\begin{equation}
\label{eq:lb_elliptic}
u_\eps(x) \ge \Big( \frac{C}{\frac{1}{u_\eps(0)^n}  + \vert x\vert^{2s-1}}   \Big)^{1/n}.
\end{equation}
Thus, in the limit $\eps\searrow 0$, the same inequality holds for the limit function
$\bar u$ that is
$$
\bar u(x) \ge \Big( \frac{C}{\frac{1}{\bar u(0)^n}  + \vert x\vert^{2s-1}}   \Big)^{1/n}.
$$
Finally, we obtain a contradiction since $n\ge 2s-1$ and $\bar u\in L^1(\R)$ by construction.

\noindent$\bullet~$\textsc{Step 3. }This case corresponds to $N=1$, $s=1/2$ and $n>0$. We use the same strategy of the corresponding parabolic problem. For any $\eps>0$ we approximate the solution of
\begin{equation}
\label{eq:ell_N1}
\begin{cases}
u_\eps + (-\Delta)^{1/2}\phi(u_\eps) = f_\eps\\
f_\eps = f+ \eps
\end{cases}
\end{equation}
with solutions of
\begin{equation}
\label{eq:ell_N1delta}
\begin{cases}
u_{\eps,\delta} + (-\Delta)^{s_\delta}\phi(u_{\eps,\delta}) = f_{\eps,\delta}\\
f_{\eps,\delta} = f_\delta+ \eps
\end{cases}
\end{equation}
where $s_\delta\searrow 1/2$ as $\delta\to 0$. This can be done thanks to
the $L^1$ contractivity of the elliptic problem, cf. \cite[Theorem 10]{pqrv2}.
Hence, for $\delta$ small enough, we have that $u_{\eps,\delta}$ satisfies the lower bound
\eqref{eq:lb_elliptic}, namely
$$
u_{\eps,\delta}(x) \ge \Big( \frac{C}{\frac{1}{u_\eps(0)^n}  + \vert x\vert^{2s_\delta-1}}   \Big)^{1/n}.
$$
Next, we let $\delta\searrow 0$ and recall that the constant $C$ does not
depend on $s$ neither on $\delta$, hence we have that $u_\eps$ satisfies the desired lower bound
$$
u_{\eps}(x)\ge \Big( \frac{C}{\frac{1}{u_\eps(0)^n}  + 1}   \Big)^{1/n}.
$$
{As a result, we have a contradiction once we assume that
the limit function $\bar u = \lim_{\eps\searrow 0}u_\eps$ is
not identically equal to zero. }

\medskip

\noindent$\bullet~$\textsc{Step 5. }To conclude the proof, we need to discuss the general case
of unbounded forcing functions in $L^1(\RN)$. This can be done exactly as in the parabolic case, through a similar density
argument.

\medskip

\noindent$\bullet~$\textsc{Step 6. }\textit{Sharpness of the range of exponents. }We only have to show that the limit maybe nontrivial for some exponents $s$  and $n$ is one space dimension. Using the idea of paper \cite{V-exist} for the parabolic problem, we only need to show that there exists some nontrivial solution or subsolution that lies below the approximations $u_\varepsilon$.

\noindent$\bullet$ In the borderline case $s=1/2$ and $n=0$, we take the spatial profile of the explicit solution \eqref{explicit.sol.mezzo},
\[
F(x)=\frac{2}{1+|x|^2}
\]
that satisfies $\Ds( \log(F)) = 2F$ in $\R$. Therefore $F$   is a nontrivial and integrable solution of the equation
\[
u+\Ds( \log(u)) = f:= 3F\,.
\]

\noindent$\bullet$ In the case $s>1/2$ and $n>0$\,, $n<2s-1$\,, we use a similar approach using a special subsolution constructed in Section 3.2 of \cite{V-exist}.        \end{proof}

\subsection{Comparison between solutions and approximate solutions}
We want now to prove non-existence for a class of weak solutions, defined as follows.
\begin{defn}
\label{def:weak_sol_elliptic}
A function $u$ is a weak solution of equation \eqref{eq:elliptic.intro} if it is nonnegative and\\[3mm]
\noindent$\bullet$ $u\in  L^1(\RN)$,
${ \phi(u)}\in\dot H^{s}(\RN)$   \\[3mm]
\noindent$\bullet$ It satisfies, for all $\zeta\in C^\infty_{0}(\RN)$,
\begin{equation}
\label{eq:elliptic_weak}
\int_{\RN}u(x)\psi(x)\d x + \int_{\RN}\Dsm \phi(u(x))\Dsm \psi(x)\d x = \int_{\RN}f(x) \psi(x)\d x.
\end{equation}
\end{defn}
 As in the parabolic problem, our goal now is to prove the following non-existence result for the above class of weak solution to the elliptic equation \eqref{eq:elliptic.intro}, which is the elliptic counterpart of Theorem \ref{th:nonex-strong}.
\begin{theor}\label{th:non_ex_elliptic}
Let $u$ denote a nonnegative weak solution of the elliptic equation \eqref{eq:elliptic.intro}, corresponding to some $f\in L^1_+(\RN)$, in the range of parameters $(N,s,n)$ allowed in Theorem $\ref{th:elliptic_1}$. Then, $u\equiv 0$ in $\RN$, hence $f=0$. This range of exponents is optimal.
\end{theor}
The proof of the above Theorem follows by combining the results of Theorem \ref{th:elliptic_1} for approximate solutions together with the following comparison result between approximate solutions and weak solutions, which is the elliptic counterpart of Theorem \ref{th:comparison}.
\begin{prop}
\label{th:comparison_elliptic}
Let $f$ be a given $L^1(\RN)$ function and let $u$ be a weak solution to the elliptic equation \eqref{eq:elliptic.intro} in the sense of Definition $\ref{def:weak_sol_elliptic}$. Then,
\begin{equation}
\label{eq:comparison_elliptic}
u\le u_\eps \,\,\,\,\hbox{ a.e. in } \RN,
\end{equation}
where $u_\eps$ is the solution \eqref{eq:elliptic_approx2} with right hand side
given by $f+\eps$.
\end{prop}
\begin{proof}
The proof is similar to the proof of Theorem \ref{th:comparison}, therefore we just sketch it.
We take the difference between the equation solved by $u$ and the equation solved by
$u_\eps$ and then we test the equation with $p_\delta(\phi(u)-\phi(u_\eps))$, with $p_\delta$ the same
smooth approximation of the Heaviside function used in the proof of Theorem \ref{th:comparison}. We have
\begin{equation*}\begin{split}
\int_{\RN}(u-u_\eps)(x) & p_\delta(\phi(u)-\phi(u_\eps)) (x)\,\d x\nonumber\\
& + \int_{\RN}\big[p_\delta(\phi(u)-\phi(u_\eps))\big]\,\Ds(\phi(u)-\phi(u_\eps)) \d x\nonumber\\
&  = \int_{\RN}(-\eps)p_\delta(\phi(u)-\phi(u_\eps))(x)\d x \le 0.
  \end{split}
\end{equation*}
The second integral in the left hand side is positive thanks
to the Strook Varopoulos inequality \eqref{Strook-Varopoulos}. \normalcolor
On the other hand, since $u$ and $u_\eps$ are in $L^1(\RN)$, we can let $\delta\searrow 0$ and obtain
\begin{equation*}
\int_{\RN}(u-u_\eps)_{+}(x)\d x \le 0,\,\,\,\,\hbox{ i.e. } u\le u_\eps \,\,\,\hbox{ for a.a. } x\in \RN,
\end{equation*}
where we have used that $\hbox{sign}_+(u-u_\eps) = \hbox{sign}_+(\phi(u)-\phi(u_\eps))$.
\end{proof}\normalcolor


\section{Non Existence for the Dirichlet problem}\label{sec.comm}

In this last Section we discuss the question of existence/non existence
for solutions of a zero Dirichlet problem for \eqref{eq:FrUF}
in a bounded domain $\Omega\subset\mathbb{R}^N$.
In this case, the scenario
for existence is identical or worse
than the Cauchy problem in the whole space. In particular, for standard diffusion (i.e. $s=1$)
solutions do not exist for integrable
initial conditions in any dimension if $m\le 0$ (see \cite{JLVSmoothing}).

We recall that  the fractional Laplacian can have different nonequivalent expressions
 on bounded domains, hence we have
 to study two different Dirichlet problems.
 There are two choices of fractional laplacian on a domain are: the Spectral and Restricted
 Laplacian , cf. \cite{BV2015, BSV2015, BV-PPR2} for the definitions.
Then, the two different Dirichlet problems are as follows
\begin{equation}
\label{eq:FrUF_diri_re}
\begin{cases}
\partial_t u + \Ds(\phi(u)) = 0 \,\,\hbox{ in } \Omega\times (0,+\infty)\\
u = 0 \,\,\hbox{ in } \mathbb{R}^N\setminus \Omega,
\end{cases}
\end{equation}
or
\begin{equation}
\label{eq:FrUF_diri_spe}
\begin{cases}
\partial_t u + \mathcal{L}_s(\phi(u)) = 0 \,\,\hbox{  in } \Omega\times (0,+\infty)\\
u = 0 \hbox{ on } \partial\Omega.
\end{cases}
\end{equation}
where we have denoted with $\mathcal{L}_s$ the Spectral laplacian (see \cite{BSV2015}).\\

Independently of the choice of the type of fractional laplacian, we have the following
non-existence result
\begin{theor}
\label{th:noex_dirichlet}
Let $\Omega$ be a bounded domain in $\mathbb{R}^N$.
Let $u$ denote a nonnegative strong solution of \eqref{eq:FrUF_diri_re}
or of \eqref{eq:FrUF_diri_spe}
in the range of parameters $(N,s,n)$
allowed in Theorems $\ref{th:non ex}$ and $\ref{th:non_ex_1D}$.
Then, $u\equiv 0$ in $\RN\times (\tau, +\infty)$ for any $\tau>0$.
\end{theor}

The above Theorem will be proved using a comparison principle. In particular, it is known that
 for smooth nonlinearities (like our $\phi_\eps$) the solutions of the Dirichlet problem
for the Spectral laplacian lie below the solutions of the Cauchy problem,
cf. e.g. \cite{pqrv1, pqrv2}, hence by comparison we deduce the same non existence results of Theorem \ref{th:nonex-strong} also for the homogeneous Dirichlet problem.

Here we discuss the non existence for the problem with the Restricted Laplacian.
As expected, the proof of the Theorem is  based on a comparison argument
for the approximate solutions
of \eqref{eq:FrUF} and \eqref{eq:FrUF_diri_re}.
More precisely,
consider $u_0\in L^1(\RN)$ and fix some compact domain $\Omega'\subset \Omega$. Then,
for any $\eps>0$
we denote with $w_\eps$ the solution of
\begin{equation}
\label{eq:diri_approx_re}
\begin{cases}
\partial_t w_\eps + \Ds \phi_\eps(w_\eps) = 0\,\,\,\hbox{ in }
\Omega\times (0,+\infty),\\
w_\eps(x,0) = \chi_{\Omega'}(x)u_0(x)\,\,\, \hbox{ in } \RN.\\
w_\eps= 0 \,\,\,\,\hbox{ in } (\RN\setminus \Omega)\times (0,+\infty),
\end{cases}
\end{equation}
where, as in the previous section, $\phi_\eps(y):=\phi(y+\eps) -\phi(\eps)$.
We recall that the Restricted fractional Laplacian is defined,
for a smooth function $v$ supported in $\Omega$ as
\begin{equation*}
\Ds v(x) :=c(N,s)\hbox{p.v.}\int_{\RN}\frac{v(x)
-v(y)}{\vert x - y\vert^{N+2s}}\d y, \,\,\,\,\,\hbox{ for } x\in \RN.
\end{equation*}
The well posedness for
\eqref{eq:diri_approx_re}
follow from \cite{pqrv2, KL, BV2015, BSV2015, BV-PPR2}.
Incidentally, note that $\bar w_\eps:=w_\eps + \eps$ solves
\begin{equation*}
\label{eq:diri_approx_re2}
\begin{cases}
\partial_t \bar w_\eps + \Ds \phi(\bar w_\eps) = 0
\,\,\,\hbox{ in } \Omega\times (0,+\infty), \\
\bar w_\eps(x,0) = \chi_{\Omega'}(x)u_0(x) + \eps\,\,\, \hbox{ in } \RN,\\
\bar w_\eps= \eps \,\,\,\,\hbox{ in } (\RN\setminus \Omega)\times (0,+\infty).
\end{cases}
\end{equation*}
Moreover, for all points in $(\RN\setminus \Omega)\times (0,\infty)$, we have that
$w_\eps(x,t)=\phi_\eps(w_\eps)=0$, hence
$$
\partial_t w_\eps(x,t) + \Ds\phi_\eps(w_\eps(x,t)) = c(N,s)\hbox{p.v.}\int_{\RN}\frac{\phi_\eps(w_\eps(x,t))
-\phi_\eps(w_\eps(y,t))}{\vert x - y\vert^{N+2s}}\d y\le 0,
$$
where $\phi_\eps$ is a positive function. Consequently, we have that $w_\eps$
verifies
$$
\partial_t w_\eps + \Ds \phi_\eps(w_\eps) \le 0 \,\,\,\hbox{ in } \RN\times (0,+\infty),
$$
with equality when $x\in \Omega$. Thus, since $\chi_{\Omega'}u_0\le u_0$,
the comparison principle gives that
$ w_\eps \le v_\eps$ in $\RN\times (0,+\infty)$ (hence $\bar w_\eps\le u_\eps$)
where $v_\eps$
is the solution of \eqref{eq:FrUF_approx1} with $u_0$ as initial condition.
On the other hand, it is not difficult to show that any weak solution
of
\begin{equation*}
\begin{cases}
\partial_t w + \Ds \phi(w) = 0 \,\,\,\hbox{ in } \Omega\times (0,+\infty)\\
w = 0 \,\,\,\hbox{ in } (\RN\setminus\Omega)\times (0,+\infty)\\
w(x,0) = \chi_{\Omega'}(x)u_0(x)\,\,\,\hbox{ in } \RN,
\end{cases}
\end{equation*}
satisfies $w\le \bar w_\eps $ in $Q$. Thus, we have the chain of inequalities
$w\le \bar w_\eps \le u_\eps$ and hence
 by letting
$\eps\searrow 0$ and recalling Theorems \ref{th:non ex} and
\ref{th:non_ex_1D} we conclude.
\normalcolor

\section{Comments and open problems}\label{sec.dir}

$\bullet$
Our result for the Dirichlet problem may not be optimal in dimension one. In fact, since our proof (both for the Spectral and  for the Restricted Laplacian) follows by a comparison argument, we actually prove that the non-existence range for the Dirichlet problem is at least as large as  the non existence range for the Cauchy problem. Is it larger? this question remains an open problem. There is some evidence that may help. On the one hand, in the cases ($N=1, s=1/2, \, n=0$) and ($N=1, s>1/2, \, n<2s-1$) the Cauchy problem has a solution (see Section \ref{sec.exist.1D} in this paper and
the recent \cite{V-exist}). On the other hand, for the Dirichlet problem in the case $N=1$ and $s=1$ (namely, standard Laplacian)  non-existence holds for any exponent $n\ge 0$.

\

\noindent{\sc Acknowledgments. } Work partially supported by Spanish Project MTM2011-24696.
Work initiated during a stay of the authors at the Isaac Newton Institute of the University of Cambridge in the spring of 2014. A.S. have been supported by Gruppo Nazionale per l'Analisi Matematica la Probabilit\`a e le loro Applicazioni (GNAMPA) of the Istituto Nazionale di Alta Matematica (INDAM).


\end{document}